\documentclass[reqno]{amsart}
\usepackage{amssymb}
\usepackage{eurosym}
\usepackage{amsfonts}
\usepackage{amsmath}
\usepackage{graphicx}
\usepackage{amscd}

\newtheorem{theorem}{Theorem}
\theoremstyle{plain}
\newtheorem{acknowledgement}{Acknowledgement}

\newtheorem{example}{Example}

\newtheorem{lemma}{Lemma}

\newtheorem{remark}{Remark}

\numberwithin{equation}{section}

\input{tcilatex}

\begin{document}
\title[Non-monotone arguments]{An oscillation criterion for delay
Differential\\
Equations with several non-monotone arguments}
\author{H. AKCA}
\address{Department of Applied Sciences and Mathematics\\
College of Arts and Sciences, Abu Dhabi University\\
Abu Dhabi, UAE}
\email{haydar.akca@adu.ac.ae, akcahy@yahoo.com}
\author{G. E. CHATZARAKIS$^{\blacktriangledown }$}
\address{Department of Electrical and Electronic Engineering Educators\\
School of Pedagogical and Technological Education (ASPETE)\\
14121, N. Heraklio, Athens, Greece}
\email{geaxatz@otenet.gr, gea.xatz@aspete.gr}
\author{I. P. STAVROULAKIS}
\address{Department of Mathematics\\
University of Ioannina\\
451 10 Ioannina, Greece}
\email{ipstav@uoi.gr}
\thanks{$^{\blacktriangledown }$Corresponding author : George E.
Chatzarakis; email address: geaxatz@otenet.gr; gea.xatz@aspete.gr; tel.
+30-210-2896774; Greece\textit{\ }}
\maketitle

\begin{abstract}
The oscillatory behavior of the solutions to a differential equation with
several non-monotone delay arguments and non-negative coefficients is
studied. A new sufficient oscillation condition, involving $\lim \sup $, is
obtained. An example illustrating the significance of the result is also
given.


\end{abstract}

\section{INTRODUCTION}

The paper deals with the differential equation with several non-monotone
delay arguments of the form%
\begin{equation}
x^{\prime }(t)+\dsum\limits_{i=1}^{m}p_{i}(t)x\left( \tau _{i}(t)\right) =0,%
\text{ \ \ }\forall t\geq 0\text{,}  \tag{1.1}
\end{equation}%
where $p_{i}$, $1\leq i\leq m$, are functions of nonnegative real numbers,
and $\tau _{i}$, $1\leq i\leq m$, are non-monotone functions of positive
real numbers such that%
\begin{equation}
\tau _{i}(t)<t\text{, \ \ }t\geq 0\text{ \ \ \ \ and \ \ \ \ }%
\lim_{t\rightarrow \infty }\tau _{i}(t)=\infty \text{, \ \ \ }1\leq i\leq m%
\text{.}  \tag{1.2}
\end{equation}

Let $T_{0}\in \lbrack 0,+\infty )$, $\tau (t)=\min \{\tau _{i}(t):i=1,\dots
,m\}$ and $\tau _{(-1)}(t)=\sup \{s:\tau (s)\leq t\}$. By a \textit{solution}
of the equation (1.1) we understand a function $x\in C([T_{0},+\infty );%
\mathbb{R}
)$, continuously differentiable on $[\tau _{(-1)}(T_{0}),+\infty )$ and that
satisfies (1.1) for $t\geq \tau _{(-1)}(T_{0})$.

A solution $x(t)$ of (1.1) is \textit{oscillatory}, if it is neither
eventually positive nor eventually negative. If there exists an eventually
positive or an eventually negative solution, the equation is \textit{%
nonoscillatory}. An equation is \textit{oscillatory} if all its solutions
oscillate.

The problem of establishing sufficient conditions for the oscillation of all
solutions of equation (1.1) has been the subject of many investigations.
See, for example, [2, 3, 5$-$13, 15,17,18] and the references cited therein.
Most of these papers concern the special case where the arguments are
nondecreasing, while a small number of these papers are dealing with the
general case where the arguments are non-monotone. See, for example, [2$,$3,
16] and the references cited therein. For the general oscillation theory of
differential equations the reader is referred to the monographs [1, 4, 14].

In 1978 Ladde [13] and in 1982 Ladas and Stavroulakis [12] proved that if 
\begin{equation}
\liminf_{t\rightarrow \infty }\int\limits_{\tau _{\max
}(t)}^{t}\sum\limits_{i=1}^{m}p_{i}(s)ds>\frac{1}{e},  \tag{1.3}
\end{equation}%
where $\tau _{\max }(t)=\max_{1\leq i\leq m}\{\tau _{i}(t)\},$ then all
solutions of (1.1) oscillate.

In 1984, Hunt and Yorke [7] proved that if $t-\tau _{i}(t)\leq \tau _{0}$, $%
1\leq i\leq m,$ and 
\begin{equation}
\liminf_{t\rightarrow \infty }\sum\limits_{i=1}^{m}p_{i}(t)\left( t-\tau
_{i}(t)\right) >\frac{1}{e},  \tag{1.4}
\end{equation}%
then all solutions of (1.1) oscillate.

When $m=1,$ that is in the special case of the equation 
\begin{equation}
x^{\prime }(t)+p(t)x\left( \tau (t)\right) =0,\ \ \forall t\geq 0, 
\tag{1.1$^{\prime }$}
\end{equation}%
in 1991, Kwong [11], proved that if 
\begin{equation*}
0<\alpha :=\liminf_{t\rightarrow \infty }\int_{\tau (t)}^{t}p(s)ds\leq 1/e,
\end{equation*}

\begin{equation}
\tau (t)\text{ is decreasing \ \ \ \ \ \ and \ \ \ \ \ \ }%
\limsup_{t\rightarrow \infty }\int\limits_{\tau (t)}^{t}p(s)ds>\frac{1+\ln
\lambda _{0}}{\lambda _{0}},  \tag{1.5}
\end{equation}%
where $\lambda _{0}$ is the smaller root of the equation $\lambda =e^{\alpha
\lambda }$, then all solutions of $(1.1)^{\prime }$ oscillate.

Recently, Braverman, Chatzarakis and Stavroulakis [2], established the
following theorem in the general case that the arguments $\tau _{i}(t)$, $%
1\leq i\leq m$ are non-monotone.

\begin{theorem}
Assume that $p_{i}(t)\geq 0$, $1\leq i\leq m$, 
\begin{equation}
h(t)=\max_{1\leq i\leq m}h_{i}(t),\text{ \ \ where \ }h_{i}(t)=\sup_{0\leq
s\leq t}\tau _{i}(s),\text{ \ \ }t\geq 0  \tag{1.6}
\end{equation}%
and $a_{r}(t,s)$, $r\in {\mathbb{N}}$ are defined as%
\begin{equation}
a_{1}(t,s):=\exp \left\{ \int_{s}^{t}\sum_{i=1}^{m}p_{i}(\zeta )~d\zeta
\right\} \text{, }a_{r+1}(t,s):=\exp \left\{
\int_{s}^{t}\sum_{i=1}^{m}p_{i}(\zeta )a_{r}(\zeta ,\tau _{i}(\zeta
))~d\zeta \right\} .  \tag{1.7}
\end{equation}%
If for some $r\in 
\mathbb{N}
$%
\begin{equation}
\limsup_{t\rightarrow \infty }\int_{h(t)}^{t}\sum_{i=1}^{m}p_{i}(\zeta
)a_{r}(h(t),\tau _{i}(\zeta ))~d\zeta >1,  \tag{1.8}
\end{equation}%
or%
\begin{equation*}
0<\alpha :=\liminf_{t\rightarrow \infty
}\int_{h(t)}^{t}\sum_{i=1}^{m}p_{i}(s)\,ds\leq 1/e,
\end{equation*}%
and%
\begin{equation}
\limsup_{t\rightarrow \infty }\int_{h(t)}^{t}\sum_{i=1}^{m}p_{i}(\zeta
)a_{r}(h(t),\tau _{i}(\zeta ))~d\zeta >1-\frac{1-\alpha -\sqrt{1-2\alpha
-\alpha ^{2}}}{2},  \tag{1.9}
\end{equation}
then all solutions of $(1.1)$ oscillate.
\end{theorem}

An oscillation criterion involving $\lim \sup $, which essentially improves
the above results is established. An example illustrating the result is also
given.

\section{MAIN RESULT}

The proof of our main result is essentially based on the following lemmas.

\begin{lemma}
$[$2, Lemma 1$]$ Assume that $x(t)$ is a positive solution of $(1.1)$ and $%
a_{r}(t,s)$ are defined by $(1.7)$. Then%
\begin{equation}
x(t)a_{r}(t,s)\leq x(s),\text{ \ \ \ \ }t\geq s\geq 0.  \tag{2.1}
\end{equation}
\end{lemma}

\begin{lemma}
$[$cf. 8$]$ Assume that $x(t)$ is a positive solution of $(1.1)$, and%
\begin{equation}
0<\alpha :=\liminf_{t\rightarrow \infty }\int_{\tau
(t)}^{t}\sum_{i=1}^{m}p_{i}(s)ds\leq \frac{1}{e}\text{,}  \tag{2.2}
\end{equation}%
where $\tau (t)=\max_{1\leq i\leq m}\tau _{i}(t)$. Then we have%
\begin{equation}
\liminf_{t\rightarrow \infty }\frac{x(h(t))}{x(t)}\geq \lambda _{0}\text{,} 
\tag{2.3}
\end{equation}%
where $h(t)$ is defined by $(1.6)$ and $\lambda _{0}$ is the smaller root of
the equation $\lambda =e^{\alpha \lambda }$.
\end{lemma}

\begin{proof}
Assume that $x(t)$ is an eventually positive solution of $(1.1)$. Then there
exists $t_{1}>0$ such that $x(t),\ x\left( \tau _{i}(t)\right) >0,\ $for all 
$t\geq t_{1}.$ Thus, from $(1.1)$ we have%
\begin{equation*}
x^{\prime }(t)=-\dsum\limits_{i=1}^{m}p_{i}(t)x\left( \tau _{i}(t)\right)
\leq 0,\text{ \ \ for all }t\geq t_{1}\text{,}
\end{equation*}%
which means that $x(t)$ is an eventually nonincreasing function of positive
numbers.

Also, by a similar procedure as in the proof of Lemma 2.1.1 [4], we have%
\begin{equation}
\liminf_{t\rightarrow \infty
}\int_{h(t)}^{t}\sum_{i=1}^{m}p_{i}(s)ds=\liminf_{t\rightarrow \infty
}\int_{\tau (t)}^{t}\sum_{i=1}^{m}p_{i}(s)ds=\alpha .  \tag{2.4}
\end{equation}%
In view of this, for any $\varepsilon \in \left( 0,\alpha \right) $, there
exists $t_{\varepsilon }\in 
\mathbb{R}
_{+}$ such that%
\begin{equation}
\int_{h(t)}^{t}\sum_{i=1}^{m}p_{i}(s)ds\geq \alpha -\varepsilon \text{ \ \ \
\ for \ }t\geq t_{\varepsilon }\geq t_{1}\text{.}  \tag{2.5}
\end{equation}%
We will show that%
\begin{equation}
\liminf_{t\rightarrow \infty }\frac{x(h(t))}{x(t)}\geq \lambda _{0}\left(
\varepsilon \right) \text{,}  \tag{2.6}
\end{equation}%
where $\lambda _{0}\left( \varepsilon \right) $ is the smaller root of the
equation%
\begin{equation*}
e^{\left( \alpha -\varepsilon \right) \lambda }=\lambda \text{.}
\end{equation*}%
Assume, for the sake of contradiction, that (2.6) is not satisfied. Then
there exists $\varepsilon _{0}>0$ such that%
\begin{equation}
\frac{e^{\left( \alpha -\varepsilon \right) \gamma }}{\gamma }\geq
1+\varepsilon _{0}\text{,}  \tag{2.7}
\end{equation}%
where%
\begin{equation*}
\gamma =\liminf_{t\rightarrow \infty }\frac{x(h(t))}{x(t)}<\lambda
_{0}\left( \varepsilon \right) .
\end{equation*}%
On the other hand, for any $\delta >0$ there exists $t_{\delta }$ such that%
\begin{equation*}
\frac{x(h(t))}{x(t)}\geq \gamma -\delta \text{ \ \ \ \ for \ }t\geq
t_{\delta }\text{.}
\end{equation*}

Dividing $(1.1)$ by $x(t)$ we obtain%
\begin{equation*}
-\frac{x^{\prime }(t)}{x(t)}=\dsum\limits_{i=1}^{m}p_{i}(t)\frac{x\left(
\tau _{i}(t)\right) }{x(t)}\geq \dsum\limits_{i=1}^{m}p_{i}(t)\frac{x\left(
h(t)\right) }{x(t)}\geq \left( \gamma -\delta \right)
\dsum\limits_{i=1}^{m}p_{i}(t)\text{.}
\end{equation*}%
Integrating last inequality from $h(t)$ to $t$ for sufficiently large $t$,
and taking into account (2.5), we have%
\begin{equation*}
-\int_{h(t)}^{t}\frac{x^{\prime }(s)}{x(s)}ds\geq \left( \gamma -\delta
\right) \int_{h(t)}^{t}\dsum\limits_{i=1}^{m}p_{i}(s)ds\geq \left( \gamma
-\delta \right) \left( \alpha -\varepsilon \right) \text{,}
\end{equation*}%
or%
\begin{equation*}
\frac{x\left( h(t)\right) }{x(t)}\geq e^{\left( \alpha -\varepsilon \right)
\left( \gamma -\delta \right) }\text{ \ \ \ \ for large }t\text{.}
\end{equation*}%
Therefore%
\begin{equation*}
\gamma =\liminf_{t\rightarrow \infty }\frac{x(h(t))}{x(t)}\geq e^{\left(
\alpha -\varepsilon \right) \left( \gamma -\delta \right) }
\end{equation*}%
which implies%
\begin{equation*}
\gamma \geq e^{\left( \alpha -\varepsilon \right) \gamma }\text{.}
\end{equation*}%
This contradicts (2.7) and therefore (2.6) is true. Thus, as $\varepsilon
\rightarrow 0$, (2.6) implies (2.3). The proof of the lemma is complete.
\end{proof}

\begin{remark}
If $\alpha >1/e$ then equation $\lambda =e^{\alpha \lambda }$ has no real
roots. In this case, lemma is inappropriate since $(1.1)$ does not have
nonoscillatory solutions at all.
\end{remark}

\begin{theorem}
Assume that $(2.2)$ holds and \textit{for some }$r\in 
\mathbb{N}
$%
\begin{equation}
\limsup_{t\rightarrow \infty }\int_{h(t)}^{t}\sum_{i=1}^{m}p_{i}(\zeta
)a_{r}(h(\zeta ),\tau _{i}(\zeta ))~d\zeta >\frac{1+\ln \lambda _{0}}{%
\lambda _{0}},  \tag{2.8}
\end{equation}%
where $h(t)$ is defined by $(1.6)$, $a_{r}(t,s)$ is defined by $(1.7)$, and $%
\lambda _{0}$ is the smaller root of the equation $\lambda =e^{\alpha
\lambda }$. Then all solutions of $(1.1)$ oscillate.
\end{theorem}

\begin{proof}
Assume, for the sake of contradiction, that there exists a nonoscillatory
solution $x(t)$ of (1.1). Since $-x(t)$ is also a solution of (1.1), we can
confine our discussion only to the case where the solution $x(t)$ is
eventually positive. Then there exists $t_{1}>0$ such that $x(t),\ x\left(
\tau _{i}(t)\right) >0,\ $for all $t\geq t_{1}.$ Thus, from (1.1) we have%
\begin{equation*}
x^{\prime }(t)=-\dsum\limits_{i=1}^{m}p_{i}(t)x\left( \tau _{i}(t)\right)
\leq 0,\text{ \ \ for all }t\geq t_{1}\text{,}
\end{equation*}%
which means that $x(t)$ is an eventually nonincreasing function of positive
numbers.

By Lemma 2, inequality (2.3) is fulfilled. Therefore%
\begin{equation}
\frac{x(h(t))}{x(t)}>\lambda _{0}-\varepsilon \text{,}\ \ \ \ \text{for all }%
t\geq t_{2}\geq t_{1}\text{,}  \tag{2.9}
\end{equation}%
where $\varepsilon $ is an arbitrary real number with $0<\varepsilon
<\lambda _{0}$. Thus, there exists a $t^{\ast }\in (h(t),t)$ such that%
\begin{equation}
\frac{x(h(t))}{x(t^{\ast })}=\lambda _{0}-\varepsilon ,\ \ \ \ \text{for all 
}t\geq t_{2}\text{.}  \tag{2.10}
\end{equation}%
Integrating (1.1) from $t^{\ast }$ to $t$ and using Lemma 1, we have%
\begin{equation*}
x(t)-x(t^{\ast })+x\left( h(t)\right) \int_{t^{\ast
}}^{t}\dsum\limits_{i=1}^{m}p_{i}(\zeta )a_{r}(h(t),\tau _{i}(\zeta ))d\zeta
\leq 0.
\end{equation*}%
Hence%
\begin{equation*}
\int_{t^{\ast }}^{t}\dsum\limits_{i=1}^{m}p_{i}(\zeta )a_{r}(h(t),\tau
_{i}(\zeta ))d\zeta \leq \frac{x(t^{\ast })}{x\left( h(t)\right) }\text{,}
\end{equation*}%
or 
\begin{equation*}
\int_{t^{\ast }}^{t}\dsum\limits_{i=1}^{m}p_{i}(\zeta )a_{r}(h(\zeta ),\tau
_{i}(\zeta ))d\zeta \leq \frac{x(t^{\ast })}{x\left( h(t)\right) }\text{,}
\end{equation*}%
which, in view of (2.10), gives%
\begin{equation}
\int_{t^{\ast }}^{t}\dsum\limits_{i=1}^{m}p_{i}(\zeta )a_{r}(h(\zeta ),\tau
_{i}(\zeta ))d\zeta \leq \frac{1}{\lambda _{0}-\varepsilon }.  \tag{2.11}
\end{equation}%
Dividing (1.1) by $x\left( t\right) ,$ integrating from $h(t)$ to $t^{\ast }$
and using Lemma 1, we have%
\begin{equation}
-\int_{h(t)}^{t^{\ast }}\frac{x^{\prime }(\zeta )}{x\left( \zeta \right) }%
d\zeta \geq \int_{h(t)}^{t^{\ast }}\dsum\limits_{i=1}^{m}p_{i}(\zeta )\frac{%
x\left( h(\zeta )\right) }{x\left( \zeta \right) }a_{r}(h(\zeta ),\tau
_{i}(\zeta ))d\zeta \text{.}  \tag{2.12}
\end{equation}%
Taking into account the fact that $\zeta \geq h(t)$ the inequality (2.9)
guarantees that%
\begin{equation}
\frac{x\left( h(\zeta )\right) }{x\left( \zeta \right) }>\lambda
_{0}-\varepsilon ,\ \ \ \ \text{for all }\zeta \geq h(t)\geq t_{2}\text{.} 
\tag{2.13}
\end{equation}%
In view of this, (2.12) gives%
\begin{equation}
-\int_{h(t)}^{t^{\ast }}\frac{x^{\prime }(\zeta )}{x\left( \zeta \right) }%
d\zeta >\left( \lambda _{0}-\varepsilon \right) \int_{h(t)}^{t^{\ast
}}\dsum\limits_{i=1}^{m}p_{i}(\zeta )a_{r}(h(\zeta ),\tau _{i}(\zeta
))d\zeta \text{,}  \notag
\end{equation}%
or%
\begin{equation*}
\int_{h(t)}^{t^{\ast }}\dsum\limits_{i=1}^{m}p_{i}(\zeta )a_{r}(h(\zeta
),\tau _{i}(\zeta ))d\zeta \leq -\frac{1}{\lambda _{0}-\varepsilon }%
\int_{h(t)}^{t^{\ast }}\frac{x^{\prime }(\zeta )}{x\left( \zeta \right) }%
d\zeta =\frac{1}{\lambda _{0}-\varepsilon }\ln \frac{x\left( h(t)\right) }{%
x(t^{\ast })}
\end{equation*}%
i.e.,%
\begin{equation}
\int_{h(t)}^{t^{\ast }}\dsum\limits_{i=1}^{m}p_{i}(\zeta )a_{r}(h(\zeta
),\tau _{i}(\zeta ))d\zeta \leq \frac{\ln (\lambda _{0}-\varepsilon )}{%
\lambda _{0}-\varepsilon }.  \tag{2.14}
\end{equation}%
Combining the inequalities (2.11) and (2.14), we have%
\begin{equation*}
\int_{h(t)}^{t}\dsum\limits_{i=1}^{m}p_{i}(\zeta )a_{r}(h(\zeta ),\tau
_{i}(\zeta ))d\zeta \leq \frac{1}{\lambda _{0}-\varepsilon }+\frac{\ln
(\lambda _{0}-\varepsilon )}{\lambda _{0}-\varepsilon }.
\end{equation*}%
The last inequality holds true for all real numbers $\varepsilon $ with $%
0<\varepsilon <\lambda _{0}$. Hence, for $\varepsilon \rightarrow 0,$ we have%
\begin{equation*}
\limsup\limits_{t\rightarrow \infty
}\int_{h(t)}^{t}\dsum\limits_{i=1}^{m}p_{i}(\zeta )a_{r}(h(\zeta ),\tau
_{i}(\zeta ))d\zeta \leq \frac{1+\ln \lambda _{0}}{\lambda _{0}},
\end{equation*}%
which contradicts (2.8). The proof of the theorem is complete.
\end{proof}

\begin{example}
Consider the delay differential equation%
\begin{equation}
x^{\prime }(t)+\frac{27}{200}x(\tau _{1}(t))+\frac{27}{200}x(\tau _{2}(t))=0,%
\text{ \ \ \ }t\geq 0\text{,}  \tag{2.15}
\end{equation}%
with%
\begin{equation*}
\tau _{1}(t)=\left\{ 
\begin{array}{ll}
t-1, & \text{if }t\in \left[ 3k,3k+1\right]  \\ 
-3t+12k+3, & \text{if }t\in \left[ 3k+1,3k+2\right]  \\ 
5t-12k-13, & \text{if }t\in \left[ 3k+2,3k+3\right] 
\end{array}%
\right. \text{ \ and \ }\tau _{2}(t)=\tau _{1}(t)-0.1\text{, \ \ \ }k\in 
\mathbb{N}
_{0}\text{,}
\end{equation*}%
where $%
\mathbb{N}
_{0}$ is the set of non-negative integers.

By $(1.6)$, we see that%
\begin{equation*}
h_{1}(t):=\sup_{s\leq t}\tau _{1}(s)=\left\{ 
\begin{array}{ll}
t-1, & \text{if }t\in \left[ 3k,3k+1\right] \\ 
3k, & \text{if }t\in \left[ 3k+1,3k+2.6\right] \\ 
5t-12k-13, & \text{if }t\in \left[ 3k+2.6,3k+3\right]%
\end{array}%
\right. \text{ \ and \ }h_{2}(t)=h_{1}(t)-0.1
\end{equation*}%
and consequently%
\begin{equation*}
h(t)=\max_{1\leq i\leq 2}\left\{ h_{i}(t)\right\} =h_{1}(t)\text{.}
\end{equation*}%
Observe that the function $f:%
\mathbb{R}
_{0}\rightarrow 
\mathbb{R}
_{+}$ defined as $f_{r}(t)=\int_{h(t)}^{t}\sum_{i=1}^{m}p_{i}(\zeta
)a_{r}(h(\zeta ),\tau _{i}(\zeta ))~d\zeta $ attains its maximum at $%
t=3k+2.6,$ $k\in 
\mathbb{N}
_{0}$, for every $r\in 
\mathbb{N}
$. Specifically,%
\begin{eqnarray*}
&&f_{1}(t=3k+2.6)=\int_{3k}^{3k+2.6}\sum_{i=1}^{2}p_{i}(\zeta )a_{1}(h(\zeta
),\tau _{i}(\zeta ))\,d\zeta \\
&=&\int_{3k}^{3k+1}\left[ p_{1}(\zeta )a_{1}(h(\zeta ),\tau _{1}(\zeta
))+p_{2}(\zeta )a_{1}(h(\zeta ),\tau _{2}(\zeta ))\right] \,d\zeta \\
&&+\int_{3k+1}^{3k+2}\left[ p_{1}(\zeta )a_{1}(h(\zeta ),\tau _{1}(\zeta
))+p_{2}(\zeta )a_{1}(h(\zeta ),\tau _{2}(\zeta ))\right] \,d\zeta \\
&&+\int_{3k+2}^{3k+2.6}\left[ p_{1}(\zeta )a_{1}(h(\zeta ),\tau _{1}(\zeta
))+p_{2}(\zeta )a_{1}(h(\zeta ),\tau _{2}(\zeta ))\right] \,d\zeta
\end{eqnarray*}%
where%
\begin{equation*}
\begin{array}{c}
\int_{3k}^{3k+1}p_{1}(\zeta )a_{1}(h(\zeta ),\tau _{1}(\zeta ))\,d\zeta
=0.135 \\ 
\\ 
\int_{3k}^{3k+1}p_{2}(\zeta )a_{1}(h(\zeta ),\tau _{2}(\zeta ))\,d\zeta
\simeq 0.138695 \\ 
\\ 
\int_{3k+1}^{3k+2}p_{1}(\zeta )a_{1}(h(\zeta ),\tau _{1}(\zeta ))\,d\zeta
\simeq 0.207985 \\ 
\\ 
\int_{3k+1}^{3k+2}p_{2}(\zeta )a_{1}(h(\zeta ),\tau _{2}(\zeta ))\,d\zeta
\simeq 0.213677 \\ 
\\ 
\int_{3k+2}^{3k+2.6}p_{1}(\zeta )a_{1}(h(\zeta ),\tau _{1}(\zeta ))\,d\zeta
\simeq 0.124791 \\ 
\\ 
\int_{3k+2}^{3k+2.6}p_{2}(\zeta )a_{1}(h(\zeta ),\tau _{2}(\zeta ))\,d\zeta
\simeq 0.128206%
\end{array}%
\end{equation*}%
Thus%
\begin{equation*}
\limsup_{t\rightarrow \infty }f_{1}(t)=\limsup_{t\rightarrow \infty
}\int_{h(t)}^{t}\sum_{i=1}^{m}p_{i}(\zeta )a_{1}(h(\zeta ),\tau _{i}(\zeta
))~d\zeta \simeq 0.948354\text{.}
\end{equation*}%
Now, we see that%
\begin{equation*}
\alpha =\liminf_{t\rightarrow \infty }\int_{\tau
(t)}^{t}\sum_{i=1}^{m}p_{i}(s)ds=\frac{27}{100}\liminf_{t\rightarrow \infty
}\left( t-\tau (t)\right) =0.27<\frac{1}{e}\text{,}
\end{equation*}%
\begin{equation*}
\liminf_{t\rightarrow \infty }\sum\limits_{i=1}^{m}p_{i}(t)\left( t-\tau
_{i}(t)\right) =\frac{27}{200}\cdot 1+\frac{27}{200}\cdot 1.1=0.2835<\frac{1%
}{e},
\end{equation*}%
and%
\begin{equation*}
\limsup\limits_{t\rightarrow \infty
}\int_{h(t)}^{t}\dsum\limits_{i=1}^{m}p_{i}(\zeta )a_{r}(h(t),\tau
_{i}(\zeta ))d\zeta \simeq 0.988865<1
\end{equation*}%
that is, none of the conditions $(1.3)$, $(1.4)$ and $(1.8)$ is satisfied.

Observe, however, that the smaller root of $e^{\alpha \lambda }=\lambda $ is 
$\lambda _{0}=1.49883$. Thus%
\begin{equation*}
0.948354>\frac{1+\ln \lambda _{0}}{\lambda _{0}}\simeq 0.937188\text{.}
\end{equation*}%
That is, condition $(2.8)$ of Theorem 2 is satisfied for $r=1$, and
therefore all solutions of $(2.15)$ oscillate.
\end{example}

\begin{acknowledgement}
The authors would like to thank both referees for the constructive remarks
which improved the presentation of the paper.
\end{acknowledgement}

\end{document}